\documentclass[12pt,twoside]{preprint}
\usepackage[margin=1.1in]{geometry}
\usepackage{amssymb}
\usepackage{amsmath}
\usepackage{hyperref}
\usepackage{breakurl}
\usepackage{mhenvs}
\usepackage{mhequ}
\usepackage{mhsymb}
\usepackage{booktabs}
\usepackage{mathrsfs}
\usepackage{times}
\usepackage{microtype}
\usepackage{comment}
\usepackage{wasysym}
\usepackage{centernot}
\usepackage[inline]{enumitem} 



\def\symbol#1{\textcolor{symbols}{#1}}
\def\1{\mathbf{\symbol{1}}}

\newtheorem{question}{Question}

\definecolor{darkred}{rgb}{0.9,0.1,0.1}

\newcommand{\e}{\varepsilon}

\def\${|\!|\!|}

\def\E{\mathbf{E}}

\def\P{\mathbf{P}}

\usepackage{graphicx}
\newcommand*{\Cdot}{{\raisebox{-0.5ex}{\scalebox{1.8}{$\cdot$}}}} 

\usepackage{acronym}
\acrodef{SHE}[SHE]{Stochastic Heat Equation}
\acrodef{KPZ}[KPZ]{Kardar--Parisi--Zhang}
\acrodef{ASEP}[ASEP]{Asymmetric Simple Exclusion Process}
\acrodef{BDG}[BDG]{Burkholder--Davis--Gundy}

\newcommand{\ic}{\text{ic}} 

\begin{document}

\title{ASEP($ q,j $) converges to the KPZ equation}
\author{Ivan Corwin$^1$, Hao Shen$^2$ and Li-Cheng Tsai$^3$}
\institute{Columbia University, US, \email{ivan.corwin@gmail.com}
\and Columbia University, US, \email{pkushenhao@gmail.com}
\and Columbia University, US, \email{lctsai.math@gmail.com}}

\maketitle

\begin{abstract}
We show that a generalized Asymmetric Exclusion Process called ASEP($ q,j $)
introduced in \cite{GASEP} converges to the Cole-Hopf solution to the KPZ equation under weak asymmetry scaling.
\end{abstract}



\section{Introduction}
In this paper we
study the generalized Asymmetric Exclusion Process called ASEP($ q,j $)
introduced in \cite{GASEP}, and show that
under the weak asymmetry scaling,
it  converges to the Cole-Hopf solution to the KPZ equation:
\begin{equ}\label{e:KPZ}
	\partial_T H
=
	\tfrac12 \Delta H + \tfrac12 ( \partial_X H )^2 + \dot W,
\end{equ}
where $\dot W$ is the space-time white noise: formally,
$\E(\dot W_T(X) \dot W_S(X'))=\delta(T-S) \delta(X-X')$.
Here the Cole-Hopf solution is defined by
$H_T(X) =  \log \mathscr Z_T(X) $
where $\mathscr Z  \in C([0,\infty),C(\R))$ is the mild solution (see \eqref{e:SHE-mild} below) to the stochastic heat equation (SHE)
\begin{equ} \label{e:SHE}
\partial_T \mathscr Z = \tfrac12 \Delta \mathscr Z + \mathscr Z \dot W\;.
\end{equ}

For the standard ASEP model, Bertini and Giacomin \cite{BG} proved its convergence in the weak asymmetry regime to the Cole-Hopf solution of the KPZ equation.
They assumed near equilibrium initial data,
and narrow wedge initial data was treated in \cite{MR2796514}.
Both of these results rely on the G{\"a}rtner transformation \cite{MR931030,dittrich91}, which is the discrete analogue of Cole-Hopf transformation. 
Recently there has been a resurgence of interest in showing that a large class of one-dimensional weakly asymmetric interacting particle system (including ASEP) should all converge to the KPZ equation. Besides the work of \cite{BG,MR2796514} (and previous to the present work), the only other result of this type via G{\"a}rtner / Cole-Hopf transform is due to \cite{DemboTsai}, 
wherein they show KPZ equation convergence for a class of weakly asymmetric non-simple exclusion processes with hopping range at most 3. 
Another work which was posted slightly after our present article is by Labb\'e \cite{labbe2016weakly,labbe2016scaling} who showed that in particular range of scaling regimes the fluctuations of the weakly asymmetric bridges converge to the KPZ equation, also via the method of G{\"a}rtner transform.

Another approach to proving KPZ equation limits for particle systems is
via energy solutions, and many microscopic models have been shown to converge to energy solutions to the KPZ equation \cite{MR3176353, goncalves15, MR3327509,MR3537337,gonccalves2015second,gonccalves2016stochastic,diehl2016kardar}, see also the lecture notes \cite{MR3445609}.
Energy solutions are proved to be unique in \cite{gubinelli15}.
The energy solution method currently only applies in equilibrium and  one needs to know the invariant measure as well as other hydrodynamic quantities explicitly.
The ASEP($q,j$) model considered presently does not have simple product form
 invariant measures, so it seems to us that the energy solution method does not apply for this model.

There are other types of systems which converge under certain weak scalings of parameters to the KPZ equation. For instance, 
\cite{AKQ,MR3189070} demonstrated KPZ convergence for the free energy of directed polymers with arbitrary disorder distributions in the intermediate  disorder regime (also called weak noise scaling).
Also,
\cite{2015arXiv1505} showed that the stochastic higher-spin vertex models introduced by \cite{CPHigherSpin} converge to KPZ under a particular weak scaling of their parameter $q\to 1$.   
The paper \cite{gubinelli2015kpz} proved the convergence of the Sasamoto-Spohn type discretizations (\cite{MR2570756}) of the KPZ/stochastic Burgers equation using paracontrolled analysis.
We also mention the recent results in the continuum setting by \cite{KPZJeremy} and \cite{CLTKPZ} using regularity structure theory, and by \cite{gubinelli2016hairer} using energy solution in the equilibrium.

The system we focus on in this paper is the ASEP($q,j$) which was introduced in \cite{GASEP} as a generalization of ASEP which allows multiple occupancy at each site (i.e., a higher spin version of ASEP). ASEP($q,j$) reduces to the usual ASEP when $j=1/2$. This class of systems was introduced through an algebraic machinery developed to construct particle systems which enjoy a certain self-duality property. The simplest case of self-duality (duality to a one-particle dual system) implies that the expectation of $q$ raised to the current of the system solves the Kolmogorov backward equation for a single particle version of the model (see Lemma 3.1 of \cite{GASEP}). This suggested to us that if we do not take expectations, the same observable might satisfy a discrete \ac{SHE}. Indeed, after writing this down, we are able to demonstrate such a discrete version of the Cole-Hopf a.k.a. G{\"a}rtner transform. We then employ methods similar to that of \cite{BG} to ultimately prove convergence of the continuum \ac{SHE}. We also remark on a similar G{\"a}rtner transform structure for the recently introduced ASIP($q,k$) \cite{CGRS2015} but do not provide a proof of convergence to KPZ for that process.

\subsection{Definition of the model and the main results}

For $q\in(0,1)$  and $ n\in\Z $, the $q$-number is defined as
\begin{equ} \label{q-num}
[n]_q=\frac{q^n-q^{-n}}{q-q^{-1}}
\end{equ}
satisfying the property $\lim_{q\to1} [n]_q = n$.
We recall the following definition of ASEP($q,j$) from \cite{GASEP}.

%
%

\begin{definition} \label{def:ASEPqj} 
Fix $q\in(0,1)$ and a half integer $j\in\N/2$.
Let $ \widetilde{\eta}(x) \in \{0,1,\ldots,2j\} $ denote
the occupation variable, i.e.\ the number of particles, at site $ x\in\Z $.
The ASEP($q,j$) is a continuous-time Markov process on the state space
$ \{0,1\ldots,2j\}^\Z = \{ (\widetilde{\eta}(x))_{x\in\Z} \} $
defined by the following dynamics:
at any given time $ t\in[0,\infty) $,
a particle jumps from site $ x $ to site $ x+1 $
at rate
\footnote{A factor $\frac{1}{2[2j]_q}$ is inserted here (comparing with \cite{GASEP}), which is unimportant but will make the coefficient in front of the Laplacian of the heat equation $\frac12$ for convenience, so that we can employ the standard heat kernel estimates.}
\begin{align*}
	\widetilde{c}_q^+(\widetilde{\eta},x)
	&= \frac{1}{2[2j]_q} q^{\widetilde{\eta}(x)-\widetilde{\eta}(x+1)-(2j+1)}
	[\widetilde{\eta}(x)]_q [2j-\widetilde{\eta}(x+1)]_q
\end{align*}
and from site $ x+1 $ to site $ x $ at rate
\begin{align*}
	\widetilde{c}_q^-(\widetilde{\eta},x)
	&= \frac{1}{2[2j]_q} q^{\widetilde{\eta}(x)-\widetilde{\eta}(x+1)+(2j+1)}
	[2j -\widetilde{\eta}(x)]_q [\widetilde{\eta}(x+1)]_q
\end{align*}
independently of each other.
With $ [0]_q = 0 $,
the property $ \widetilde{\eta}(x) \in \{0,1,\ldots,2j\} $
is clearly preserved by the dynamics described in the proceeding,
and with $ \widetilde{c}_q^\pm(\Cdot,\Cdot) $ being uniformly bounded,
such a process is constructed by the standard procedures as in \cite{liggett12}.
\end{definition}

Focusing on the fluctuation around density $ j $,
we define the centered occupation variable
$ \eta(x) := \widetilde{\eta}(x) -j \in \{-j,\ldots,j\} $
and the corresponding jumping rate
\begin{align}\label{e:rates}
\begin{split}
	c_q^+(\eta,x)&= \frac{1}{2[2j]_q} q^{\eta(x)-\eta(x+1)-(2j+1)}
	[j+\eta(x)]_q [j-\eta(x+1)]_q 
\\
	c_q^-(\eta,x) &= \frac{1}{2[2j]_q} q^{\eta(x)-\eta(x+1)+(2j+1)}
	[j -\eta(x)]_q [j+\eta(x+1)]_q \;.
\end{split}
\end{align}
Under these notations, the ASEP$(q,j)$ has the generator
\begin{equ}\label{genZ}
({\cal L}f)(\eta) =\sum_{x\in \Z} ({\cal L}_{x,x+1}f)(\eta)
\end{equ}
where
\begin{equ} \label{gen}
({\cal L}_{x,x+1}f)(\eta)
 =  c^+_q (\eta, x) \, \big(f(\eta^{x,x+1}) - f(\eta)\big)
 + c^-_q (\eta, x) \, \big(f(\eta^{x+1,x}) - f(\eta)\big)
\end{equ}
and $ \eta^{x,y} $ is the configuration obtained by
moving a particle from site $x$ to site $y$.


For any function $f:\Z \to \R $,
define the forward and backward discrete gradients as
\[
\nabla^+ f(x) \eqdef f(x+1)-f(x) \;, \qquad
\nabla^- f(x) \eqdef f(x-1)-f(x) \;.
\]
Define the height function $h$ so that $\nabla^+ h(x)=\eta(x+1)$.
More precisely, let
$h_t(0)$ be the net flow of particles from $x=1$ to
$x=0$ during the time interval $[0,t]$,
counting \emph{left-going} particles as positive,
and
\begin{align}
	\label{e:def-h}
	h_t(x) & \eqdef
	h_t(0) +
	\left\{\begin{array}{c@{,}l}
		\sum_{0<y \le x} \eta_t(y)	& \text{ when } x \ge 0 \;,\\
		-\sum_{x<y\le 0} \eta_t(y)	& \text{ when } x<0 \;.
	\end{array}\right.
\end{align}
We define the microscopic Hopf-Cole / G{\"a}rtner transform
of the height function $ h_t(x) $ as
\begin{equ}
	\label{e:Z-defn}
	Z_t(x) \eqdef q^{-2 h_t(x)+\nu t} \;
\end{equ}
where the term $ \nu t $ is to balance the overall linear (in time) growth of $ h_t(x) $, with
\begin{align}\label{nu}
	\nu \eqdef \big( \tfrac{[4j]_q}{2[2j]_q} -1 \big) / \ln q\;.
\end{align}
We linearly interpolate $ Z_t(x) $ in $ x\in\R $ so that $ Z \in D([0,\infty),C(\R)) $,
the space of $ C(\R) $-valued, right-continuous-with-left-limits processes.

Turning to our main result,
we consider the weakly asymmetric scaling $ q = q_\eps = e^{-\sqrt \eps} $, $ \eps \to 0 $,
whereby $ \nu = \nu_\eps = -2j^2 \sqrt\eps +O(\eps)$.
To indicate this scaling,
we denote \emph{parameters} such as $ \nu $ by $ \nu_\e $,
but for \emph{processes} such as $ h_t(x) $ and $ M_t(x) $,
we often omit the dependence on $ \eps $ to simplify notations.
Following \cite{BG}, we consider the following
near equilibrium initial conditions:
\begin{definition}\label{def:nearEq}
Let $\|f_t(x)\|_n \eqdef (\E|f_t(x)|^n)^{\frac1n}$ denote the $L^n$-norm.
We say a sequence $ \{h^\eps_0(\Cdot)\}_\e $
of initial conditions is near equilibrium if,
for any $\alpha \in(0,\frac12)$ and every $n\in\N$
there exist finite constants $C$ and $a$ such that
\begin{equs}
	\label{e:init-uniform}
	\Vert Z_0(x) \Vert_n
	&\le C e^{a\eps |x|} \;,
	\\
	\label{e:init-Holder}
	\Vert Z_0(x)-Z_0(x')\Vert_n
	&\le C (\eps|x-x'|)^\alpha e^{a\eps(|x|+|x'|)} \;.
\end{equs}
\end{definition}
%
%
%
%

%
Recall that $ \mathscr Z_T(X) $ is the solution to the \ac{SHE} \eqref{e:SHE}
starting from $ \mathscr Z_0(\Cdot) \in C(\R) $ if
\begin{equ} \label{e:SHE-mild}
 \mathscr Z_T = \mathscr P_T *\mathscr Z_0
 	 + \int_0^T \mathscr P_{T-S} * (\mathscr Z_S \, dW_S)
\end{equ}
where $ \mathscr P$ is the standard heat kernel,
and the last integral is in It\^o sense and $*$ denotes the spatial convolution.
Hereafter, we endow the space $D([0,\infty),C(\R))$ with the Skorokhod topology
and the space $ C(\R) $ with the topology of uniform convergence on compact sets,
and use $ \Rightarrow $ to denote weak convergence of probability laws.
Write $\eps_j\eqdef 2j\eps$ and consider the scaled processes
\begin{equ}\label{e:calZ}
	\mathcal Z_{T}^\eps(X)
	\eqdef Z_{\eps_j^{-2} T}(\eps_j^{-1} X) \in D([0,\infty),C(\R)) \;.
\end{equ}
The following is our main theorem.

\begin{theorem}\label{thm:main}
Let $ \mathscr Z^\ic \in C(\R) $ and $ \mathscr Z $
be the unique solution to \ac{SHE} from $ \mathscr Z^\ic $.
Given any near equilibrium initial conditions
such that $ Z^\e_0 \Rightarrow \mathscr Z^\ic $, as $ \e \to 0 $,
under the preceding weakly asymmetric scaling,
we have that $ \mathcal Z^\eps \Rightarrow \mathscr Z $,
as $ \eps \to 0 $.
\end{theorem}

Definition~\ref{def:nearEq} (and therefore Theorem~\ref{thm:main})
leaves out an important initial condition,
i.e.\ the step initial condition: 
\begin{equ} [e:step-init]
 \eta_0(x) = j  \quad \mbox{for }  x\leq 0 \;, \qquad
\mbox{and }  \eta_{0}(x) = -j  \quad \mbox{for }  x >0  \;.
\end{equ}
Following \cite{MR2796514}, we generalize Theorem~\ref{thm:main} to the following:

\begin{theorem}\label{thm:step}
Let $ \mathscr Z^* $ be the unique solution of \ac{SHE}
starting from the delta measure $ \delta(\Cdot) $,
let $ \{\eta_0(x)\}_{x} $ the step initial condition as in \eqref{e:step-init},
and let $ \mathcal Z^{*,\eps}_T(X) := \frac{1}{2\sqrt \eps} \mathcal Z_{T}^\eps(X)  $.
We have that $ \mathcal Z^{*,\eps} \Rightarrow \mathscr Z^* $,
as $ \eps \to 0 $.
\end{theorem}

\subsection{Proof of Theorems~\ref{thm:main} and \ref{thm:step}}

In Section~\ref{sect:tight}, we establish the following moment estimates.
\begin{proposition} \label{prop:Holder}
Fix $ \bar T <\infty $, $ n \in \N $, $ \alpha\in(0,1/2) $,
and some near equilibrium initial conditions as in Definition~\ref{def:nearEq},
with the corresponding finite constant $ a $.
Then, there exists some finite constant $ C $ such that
\begin{equs}
	\label{e:uniform}
	\Vert Z_t(x) \Vert_{2n} &\leq C e^{a\eps |x|}
\\
	\label{e:Holderx}
	\Vert Z_t(x)-Z_t(x')\Vert_{2n} &\le
	C (\eps |x-x'|)^\alpha e^{a\eps (|x|+|x'|)}
\\
	\label{e:Holdert}
	\Vert Z_t(x) - Z_{t'}(x) \Vert_{2n} &\le
	 C (1\vee|t'-t|^\frac{\alpha}{2}) \eps^\alpha e^{2a\eps |x|}
\end{equs}
for all $t,t'\in[0,\eps_j^{-2} \bar T]$ and $ x,x'\in\R $.
\end{proposition}
Applying the argument as in \cite[Proof of Proposition~1.4]{DemboTsai}
(see also \cite[Proof of Theorem~3.3]{BG}),
we then have that Proposition~\ref{prop:Holder} implies the following tightness result
\begin{proposition}\label{prop:tight}
For near equilibrium initial conditions,
the law of $ \{\mathcal Z^\eps\}_\eps $
is tight in $ D([0,\infty)\times\R) $.
Moreover, limit points of $ \{\mathcal Z^\eps\}_\eps $ concentrates
on $ C([0,\infty)\times C(\R)) $.
\end{proposition}
With this, in Section~\ref{sect:unique} we prove the following proposition,
which, together with the uniqueness of the \ac{SHE}, 
completes the proof of Theorem~\ref{thm:main}.
\begin{proposition}\label{prop:unique}
For near equilibrium initial conditions,
any limiting point $ \mathscr Z $ of $ \{\mathcal Z^\e \}_\e $ solves the \ac{SHE}.
\end{proposition}

Turning to Theorem~\ref{thm:step},
with $ \mathcal Z^{\eps,*} $ as in Theorem~\ref{thm:step},
we have that
\[
	\lim_{\e \to 0} \e \sum_{x\in\Z} Z^{\eps,*}_0(x) \to 1.
\]
Combining this with the exponential decay (in $ |x| $) of $ Z^{\eps,*}_0(x) $,
one easily obtains $ Z^{\eps,*}_0(\Cdot) \Rightarrow \delta(\Cdot) $.
With this and Theorem~\ref{thm:main},
following the argument of \cite[Section 3]{MR2796514}
Theorem~\ref{thm:step} is an immediate consequence of
the following moment estimates,  which we establish in Section~\ref{sect:tight}.
\begin{proposition}\label{prop:stepEst}
Let $ Z^*_t(x) = \frac{1}{2\sqrt{\eps}} Z_t(x) $.
For the step initial condition,
for any $ T<\infty $, $ n\geq 1 $ and $ \alpha\in(0,1/2) $,
there exists $ C $ such that
\begin{align}
	&
	\label{eq:momd}
	\Vert {Z}^{*}_t(x) \Vert_{2n} \leq C / \sqrt{\eps^2 t },
\\
	&
	\label{eq:momdx}
	\Vert {Z}^{*}_t(x) - {Z}^{*}_t(x) \Vert_{2n}
	\leq C
	(\e|x-x'|)^{\alpha} (\eps^2 t)^{-(1+\alpha)/2},
\end{align}
for all $ t\in(0,\eps_j^{-2}T] $ and $ x,x'\in\R $.
\end{proposition}

\subsection{Outline}

The rest of the paper is organized as follows. In Section~\ref{sect:SHE}
we show the crucial result that for the ASEP($q,j$) model,
one can still achieve the discrete Hopf-Cole / G{\"a}rtner transform.
In Section~\ref{sect:tight} we prove tightness of the rescaled processes 
as in the ASEP case in \cite{BG}, but we use some of the more recent treatments
in \cite{DemboTsai} which simplified the arguments of \cite{BG}.
In Section~\ref{sect:unique} we identify the limit as the solution of SHE; 
which essentially follows  the arguments of \cite{BG} but in the ``key estimate" we provided a proof to the more general case of a crucial cancellation and since  \cite{BG} was written twenty years ago, we make the proofs slightly more streamlined in our presentation.

\subsection{Acknowledgements}

IC was partially supported by the NSF DMS-1208998, 
by a Clay Research Fellowship, and by a Packard Fellowship for Science and Engineering. 
LCT was was partially supported by the NSF DMS-1106627 and the KITP graduate fellowship.
Some of this work was done during the Kavli Institute for Theoretical Physic program 
``New approaches to non-equilibrium and random systems: KPZ integrability, universality, applications and experiments'' and was supported in part by the NSF through NSF PHY11-25915.

\section{Microscopic \ac{SHE}}
\label{sect:SHE}


In this section we derive the microscopic Hopf-Cole / G{\"a}rtner transform of ASEP($ q,j $),
stated in the  following proposition. This discrete level Hopf-Cole transformation was introduced by G{\"a}rtner \cite{MR931030}, see also \cite{dittrich91} by Dittrich and G{\"a}rtner.
\begin{proposition}\label{prop:HC}
\begin{enumerate}[label=(\alph*)]
	\item[]
	\item\label{enu:HCdr}
	For  $ Z_t(x) $ is defined as in \eqref{e:Z-defn},
	we have that
	\begin{equ}\label{duality}
		d Z_t(x) = \tfrac12 \Delta Z_t(x)\,dt + dM_t(x)
	\end{equ}
	where $\Delta f(x) = f(x+1) + f(x-1) - 2f(x) $ denotes the standard discrete Laplacian
	and $ M_\Cdot(x) $, $ x\in\Z $, are martingales;
	\item\label{enu:HCMG}
	Furthermore, for the martingale term, we have
	\begin{equ} \label{e:ASEPbrac}
	\frac{d}{dt}  \langle M(x),M(y) \rangle_t
	 = \mathbf{1}_{\{x=y\}} \Big(\tfrac{4\eps j^2}{[2j]_q} Z_t(x)^2
		 -\tfrac{1}{[2j]_q}\nabla^+ Z_t(x) \nabla^- Z_t(x) +o(\eps)Z_t(x)^2 \Big)
	\end{equ}
	where $o(\eps)$ is a term uniformly bounded by constant $C_\eps$ and
	$C_\eps / \eps \to 0$.	
\end{enumerate}
\end{proposition}

To simplify notations, throughout
this section we \emph{omit} the dependence of parameters (e.g.\ $ q, \nu $) on $ \eps $.
To prove Proposition~\ref{prop:HC},
we note that each jump from $ x $ to $ x+1 $ (resp.\ from $ x+1 $ to $ x $)
decreases (resp.\ increases) $ h(x) $ by $ 1 $.
Taking into account the factor $ q^{\nu t} $ in \eqref{e:Z-defn},
we obtain from \eqref{genZ} that
%
\begin{align}
	\notag
	dZ_t(x)
	=
	(q^{2}-1) Z_t(x) c^+(\eta, x ) dP_t^+ (x)
	&+ (q^{-2}-1) Z_t(x) c^-(\eta, x) dP_t^-(x)
\\
	\label{eq:HC:}
	&+ Z_t(x) \nu \ln q \,dt
\end{align}
where $\{P_t^+(x)\}_{x\in\Z}$ and $\{P_t^-(x)\}_{x\in\Z}$
are independent Poisson processes with unit rate.
Letting $ M_t^\pm(x) := \int_0^t ( c^\pm(\eta(s), x) dP_s^\pm(x) - c^\pm(\eta(s), x) ds)$
denote the corresponding compensated Poisson processes,
which is a martingale,
we have that
\begin{equ}
	dZ_t(x) =  \Omega Z_t(x)  \,dt
			+ dM_t(x)
\end{equ}
where the drift term has coefficient
\begin{equ} \label{e:Omega}
\Omega= (q^{2}-1)  c^+ (\eta,x)
			+(q^{-2}-1) c^-(\eta,x) +\nu \ln q
\end{equ}
and the martingales $\{M_t(x)\}_{x\in \Z}$  are defined as
\begin{equ} \label{e:M}
	M_t(x) =
	\int_0^t \big(
		(q^{2}-1) Z_s(x) dM_s^+(x)
		+ (q^{-2}-1) Z_s(x) dM_s^-(x) \big) \;.
\end{equ}

\begin{proof}[of Proposition~\ref{prop:HC}\ref{enu:HCdr}]

With \eqref{eq:HC:},
proving \eqref{duality} amounts
to proving $ \Omega Z_t(x) = \frac{1}{2} \Delta Z_t(x) $.
First of all, by the definition \eqref{e:Z-defn} of $Z_t$,
we clearly have (omitting the subscript $t$ for simplicity):
\begin{equ}\label{Lap:Z}
	\Delta Z(x) = \Big(q^{-2 \eta (x+1)}+q^{2 \eta (x)} -2 \Big) Z(x) \;.
\end{equ}
On the other hand, by straightforward computation
using the definition \eqref{e:Omega} of $\Omega$
and the expression \eqref{e:rates} of the rates $c^\pm$,
\begin{equs}
2[2j]_q(\Omega &-\nu \ln q )
=  (q^{2}-1)  \,
	q^{\eta(x)-\eta(x+1)-(2j+1)} \, [j+\eta(x)]_q [j-\eta(x+1)]_q \\
 &\qquad\qquad  +(q^{-2}-1) \,
 	q^{\eta(x)-\eta(x+1)+(2j+1)}\, [j -\eta(x)]_q [j+\eta(x+1)]_q \\
&=  (q^{2}-1)  \,q^{\eta(x)-\eta(x+1)-(2j+1)}
	\,\frac{q^{j+\eta(x)} -q^{-(j+\eta(x))}}{q-q^{-1}}
	 \,\frac{q^{j-\eta(x+1)} -q^{-(j-\eta(x+1))}}{q-q^{-1}} \\
 &\quad +(q^{-2}-1) \, q^{\eta(x)-\eta(x+1)+(2j+1)}
 	\,\frac{q^{j-\eta(x)} -q^{-(j-\eta(x))}}{q-q^{-1}}
	 \,\frac{q^{j+\eta(x+1)} -q^{-(j+\eta(x+1))}}{q-q^{-1}} \\
&=  \frac{1}{q-q^{-1}}
	\Big( \big(q^{2\eta(x)}-q^{-2j} \big)\,
		\big(q^{-2\eta(x+1)}-q^{-2j} \big)
	 -\big(q^{2j}-q^{2\eta(x)} \big) \, \big(q^{2j} - q^{-2\eta(x+1)} \big) \Big) \\
&=  \frac{1}{q-q^{-1}} \Big( q^{-4j} -q^{4j} +q^{2\eta(x)+2j} -q^{2\eta(x)-2j}
		+q^{2j-2\eta(x+1)} -q^{-2j-2\eta(x+1)}\Big) \\
&=[2j]_q \Big(q^{2\eta(x)} +  q^{-2\eta(x+1)}-2\Big) -[4j]_q+2[2j]_q
\end{equs}
Comparing this with \eqref{Lap:Z} one obtains
\begin{align*}
	\Omega Z_t(x) = \tfrac12 \Delta Z_t(x)
	+ \Big(\nu\ln q+1 -[4j]_q/(2[2j]_q)\Big) Z_t(x).
\end{align*}
With this and \eqref{nu}, the desired result
$ \Omega Z_t(x) = \frac{1}{2} \Delta Z_t(x) $ follows.
\end{proof}
%
%
%
%

%
%

\begin{proof}[of Proposition~\ref{prop:HC}\ref{enu:HCMG}]
By the definition \eqref{e:M}, the bracket process of $M_t$ is
\begin{equ}
\frac{d}{dt}\langle M(x),M(y) \rangle_t
	=  \mathbf{1}_{\{x=y\}} \Big((q^{2}-1)^2  c^+(\eta_t,x)
			+(q^{-2}-1)^2  c^-(\eta_t, x) \Big) Z_t(x)^2
\end{equ}
For ASEP(q,j),
substituting $c^\pm $ and following similar computations as above,
and by independence of $M(x)$ and $M(y)$ for $x\neq y$, one has
\begin{equs}
\frac{d}{dt}&\langle M(x),M(y) \rangle _t
=\frac{\mathbf{1}_{\{x=y\}}}{2[2j]_q} Z_t(x)^2
\\
&
	\times \Big(  \frac{q^2 -1}{q-q^{-1}}
		\big(q^{2\eta(x)}-q^{-2j} \big)\,\big(q^{-2\eta(x+1)}-q^{-2j} \big)
	   -\frac{q^{-2} -1}{q-q^{-1}}
	   	\big(q^{2j}-q^{2\eta(x)} \big)\,
	   	\big(q^{2j} - q^{-2\eta(x+1)} \big) \Big)
\end{equs}
With $q=e^{-\sqrt{\eps}}$, $q^a= 1-a\sqrt\eps +o(\sqrt\eps)$,
for any uniformly bounded variable $a$,
we further obtain
\begin{equs}
\frac{d}{dt} & \langle M(x),M(y) \rangle_t \\
& = - \frac{2\eps\mathbf{1}_{\{x=y\}}}{[2j]_q} Z_t(x)^2
	\Big(  \big(\eta(x) +j\big)\,\big(\eta(x+1)-j\big)
		+ \big(\eta(x)-j \big) \,\big(\eta(x+1)+j \big) +o(1)\Big) \\
& =  \frac{4\eps \mathbf{1}_{\{x=y\}}}{[2j]_q} Z_t(x)^2
	\Big(  j^2-\eta(x) \eta(x+1)+o(1)\Big) \label{e:another-brac} \;.
\end{equs}
On the other hand,
\begin{equs}
\nabla^+ Z_t(x) &= (q^{-2\eta(x+1)}-1)Z_t(x)
=\big(2\eta(x+1)\sqrt\eps +o(\sqrt\eps )\big)Z_t(x)  \\
\nabla^- Z_t(x) &= (1-q^{2\eta(x)})Z_t(x)
=\big(2\eta(x)\sqrt\eps +o(\sqrt\eps )\big)Z_t(x),
\end{equs}
from which the desired result \eqref{e:ASEPbrac} follows.
\end{proof}

A useful bound on $ \frac{d}{dt}\langle M(x),M(x) \rangle_t $ is the following
\begin{corollary}
For $ M_t(x) $ as in Proposition~\ref{prop:HC}, we have that
\begin{equ}\label{eq:QVbd}
	|\tfrac{d}{dt} \langle M(x),M(y) \rangle_t|
	\le \mathbf{1}_{\{x=y\}} C\eps Z_t(x)^2  \;
\end{equ}
for some finite constant $ C $.
\end{corollary}
\begin{proof}
This follows directly from \eqref{e:another-brac} and the boundedness of $ \eta_t(x) $.
\end{proof}

\begin{remark}
The same term $\nabla^+ Z_t(x) \nabla^- Z_t(x) $ as in \eqref{e:ASEPbrac} also appears in \cite[Eq.(3.15)]{BG}. The appearance of this term
indicates that we will need to adapt the ``key estimate"
in \cite[Lemma 4.8.]{BG} to our case.
Note also that if $j=\frac12$,  the coefficient of $Z_t(x)^2$ in $\frac{d}{dt}  \langle M(x),M(x) \rangle_t $
is nearly $\eps$, the same  with \cite{BG}.
\end{remark}
%
%
%

\section{Tightness, proof of Propositions~\ref{prop:Holder} and \ref{prop:stepEst}}
\label{sect:tight}

\begin{lemma}\label{lem:BDG}
Given any $ n\in\N $,
there exists a finite constant $ C $ such that,
for any deterministic function $f_s(x,x')$: $[0,\infty)\times\Z^2\to\R$
and any $ t\leq t' \in [0,\infty) $ with $ t'-t \geq 1 $,
\begin{equ}
	\Big\Vert
	\int_t^{t'}  f_s * dM_s (x) \Big\Vert_{2n}^2
\le
	C \eps \int^{ t' }_{ t }
	 \bar f_{s}^{\;2} * \Vert Z_s^2 \Vert_n (x) \,ds
\end{equ}
where
$\bar f_{s}(x,x')\eqdef \sup_{|s'-s| \leq 1 } | f_{s'}(x,x')| $.
\end{lemma}

\begin{proof}
This proof is essentially by \cite[Lemma~3.1]{DemboTsai},
which we adapt into our setting.
Fix such $t, t' $ and let
$R_{t'}(x):=\int_t^{t'}  f_s * dM_s (x)$.
By the \ac{BDG} inequality,
\begin{equ}\label{e:applyBDG}
	\Vert R_{t'}(x)^2 \Vert_n
\le
	C \Vert [ R_\Cdot(x) ]_{t'} \Vert_n,
\end{equ}
where $[-]$ denotes the optional quadratic variation,
or more explicitly
\begin{equ}\
	[R_\Cdot (x)]_{t'}
=
	\sum_{x'}
	\sum_{s\in\mathfrak T(x')} f_{s^{-}}(x,x')^2
	(q^{\pm 2}-1)^2 Z_{s^{-}}(x')^2
\end{equ}
where $\mathfrak T(x')$ is the set of $s\in(t,t']$
at which a jump  occurs at the site $x'$, and the $\pm$ is dictated by the direction of the jump.
Next, letting $ k := \lceil t'-t \rceil $,
we partition $ (t,t'] $ into subintervals $\mathcal T_i = (t_{i-1}, t_i] $,
where $ t_i \eqdef t+(t'-t)\frac{i}{n} $.
Each $ \mathcal T_i $ has length $ \frac12 \leq |\mathcal T_i| \leq 1 $.
Using $|q^{\pm 2}-1| \le C \sqrt \eps $,
and replacing $f_s$ and $Z_s$ by their supremum over $\mathcal T_i$, we have
\begin{equ}\ 
    [R_\Cdot(x)]_{t'}
\le
     C \eps \sum_{i=1}^{n} \sum_{x'} N_{\mathcal T_i}(x')
	\bar f_{t_{i-1}}(x,x')^2
	\Big( \sup_{s\in\mathcal T_i}  Z_s(x')^2 \Big)
\end{equ}
where $N_I(x')$ is the number of jumps at $x'$ during the time interval $ \mathcal T_I $.
Further using
\begin{align}\label{eq:Z:ss}
	\sup_{s\in (s_1,s_2]} Z_s(x') \leq e^{2 \sqrt{\eps} N_I(x')} Z_{s_1}(x') \;,
\end{align}
the independence of $ N_{\mathcal T_i} $ and $ Z_{t_{i-1}}(x') $,
and the fact that $N_{\mathcal T_i}(x')$
is stochastically bounded by a Poisson random variable with constant rate,
one obtains the desired bound.
\end{proof}

%

Let $ R(t) $ be the continuous time random walk on $ \Z $, starting from $ x=0 $,
which jump symmetrically  $ \pm 1 $ step at rate $ \frac12 $.
Let $ p_t(x) = \P( R(t)=x ) $ denote the corresponding heat kernel.
We rewrite the discrete \ac{SHE} \eqref{duality}
in the following integrated form:
\begin{equ} \label{e:intDSHE}
	Z_t = p_t * Z_0 + \int_0^t p_{t-s} * dM_s \;.
\end{equ}

\begin{proof}[of Proposition~\ref{prop:Holder}]
Let $I_1$ and $I_2$ denote the first and second terms
on the RHS of \eqref{e:intDSHE}, respectively.

We begin by proving \eqref{e:uniform}.
First, by \cite[(A.24)]{DemboTsai} we have
the following bound on the standard heat kernel
\begin{equ} \label{e:ev-exp}
	(p_t * e^{a\eps |\Cdot|}) (x) \le C e^{a\eps |x|} \qquad \mbox{for } t\le \eps^{-2} \bar T \;.
\end{equ}
For $I_1$, by the triangle inequality we have
$
	\Vert I_1(t,x)^2\Vert_n = \Vert I_1(t,x)\Vert_{2n}^2
	\le ( p_t * \Vert Z_0(\Cdot)\Vert_{2n} (x))^2.
$
Combining this with \eqref{e:ev-exp} and \eqref{e:init-uniform}, we obtain
\begin{align}\label{eq:I1}
	\Vert I_1(t,x)^2 \Vert_{n} \leq C e^{2a\eps|x|}.
\end{align}
Turning to bounding $ I_2 $,
we assume $ t \geq 1 $ and apply Lemma~\ref{lem:BDG} with $ f_s(x,x')=p_{t-s}(x-x') $ to obtain
\begin{equ}
	\Vert I_2(t,x)^2 \Vert_n
\le
	C \eps \int_0^t
	 \bar p_{ t-s}^2*\Vert Z_s^2 \Vert_n (x) ds
\end{equ}
where $\bar p$ is the local supremum of $p$ defined as in Lemma~\ref{lem:BDG}.
By $p_t \le Cp_{t'}$ for $|t-t'|\le 1$ and the standard heat kernel estimate
$p_t \le Ct^{-\frac12}$,
\begin{equ}
	\Vert I_2(t,x)^2 \Vert_n
\le
	C \eps \int_0^t (t-s)^{-\frac12}
	\Big( p_{ t-s }*\Vert Z_s^2 \Vert_n (x)\Big) ds\;,
	\text{ for } t \geq 1\;.
\end{equ}
Combining this with \eqref{eq:I1} yields
\begin{equ}\label{eq:iter}
	\Vert Z_t^2(x) \Vert_{n}
	\leq
	C e^{2a\eps|x|}
	+ C \eps \int_0^t (t-s)^{-\frac12} \Big( p_{ t-s }*\Vert Z_s^2 \Vert_n (x)\Big) ds\;.
\end{equ}
The bound \eqref{eq:iter} was derived for $ t \geq 1 $,
but it in fact holds true also for $ t \leq 1 $.
This is so because, by \eqref{e:uniform} and \eqref{eq:Z:ss} with $ (s_1,s_2]=(0,t] $,
we already have $ \Vert Z^2_t(x) \Vert_{2n} \leq C e^{2a\eps|x|} $, for $ t \leq 1 $.
With this, iterating this inequality,
using the semi-group property $ p_s* p_{s'} = p_{s+s'} $ and \eqref{e:ev-exp},
we then arrive at
\begin{equ}
	\Vert Z_t^2(x) \Vert_{n}
	\leq
	\Big(
		C e^{2a\eps|x|}
		+ \sum_{j=1}^\infty
		 \frac{C^j}{j!} \Big(
		\eps\int_0^t s^{-1/2} ds \Big)^j
		e^{2a\eps|x|}
	\Big) \;.
\end{equ}
With $t\le \eps^{-2}\bar T$, the desired result \eqref{e:uniform} follows.

The bound \eqref{e:Holderx} is proved analogously.
Indeed,
\begin{equ}
 \Vert I_1(t,x) -I_1(t,x') \Vert_{2n}^2
	\le \Big( \sum_{\bar x} p_t(\bar x)
		\Vert Z_0(x-\bar x)-Z_0(x'-\bar x)\Vert_{2n}  \Big)^2 \;.
\end{equ}
By \eqref{e:init-Holder}, followed again by \eqref{e:ev-exp},
the preceding expression is bounded by
\begin{equ} \label{e:I1Holderx}
 \Big( \sum_{\bar x} p_t(\bar x) \,
		(\eps|x-x'|)^\alpha e^{a\eps(|x-\bar x|+|x'-\bar x|)}   \Big)^2
\le
 (\eps|x-x'|)^{2\alpha} e^{2a\eps(|x|+|x'|)} \;.
\end{equ}
For $\Vert I_2(t,x) -I_2(t,x') \Vert_{2n}^2$,
we apply Lemma~\ref{lem:BDG}
with $f_s(x,\bar x)=p_{t-s}(x'-\bar x) -p_{t-s}(x-\bar x)$,
use the fact that
\[
\big( p_{t-s}(x'-\bar x) -p_{t-s}(x-\bar x) \big)^2
\le
\big| p_{t-s}(x'-\bar x) -p_{t-s}(x-\bar x) \big| \,
\big( p_{t-s}(x'-\bar x) + p_{t-s}(x-\bar x) \big)
\]
and use the gradient estimate for the heat kernel, for instance \cite[(A.10)]{DemboTsai}:
\[
\big| p_{t-s}(x'-\bar x) -p_{t-s}(x-\bar x) \big|
	\leq
	C (1\wedge (t-s)^{-\frac{1}{2}-\frac{\alpha}{2}}) |x-x'|^\alpha \;.
\]
The rest of the arguments follow in the same way as the proof for \eqref{e:uniform}.

Next we prove \eqref{e:Holdert}.
Without lost of generality, we assume $ t < t'-1 $.
For $I_1$, using the semi-group properties $p_{t'}=p_{t'-t}*p_t$
and $\sum_{x_1} p_{t'-t}(x_1)=1$ we have
\begin{equ}\label{eq:I1:tt}
	I_1(t',x) - I_1(t,x) = \sum_{\bar x} p_{t'-t}(x-\bar x)(I_1(t,\bar x)-I_1(t,x)).
\end{equ}
By \eqref{e:I1Holderx}, we have
$
	\Vert I_1 (t,\bar x)-I_1 (t,x) \Vert_{2n}
	\le C
	(\eps|x-\bar x|)^\alpha e^{a\eps |x-\bar x|} e^{2a \eps |x|}
	 .
$
Using this and the estimate
$
	\sum_x |x|^\alpha p_{t'-t}(x) e^{a\eps |x|}
	\leq
	C |t'-t|^{\frac{\alpha}{2}}
$
in \eqref{eq:I1:tt},
one obtains the desired bound on $ \Vert I_1 (t,\bar x)-I_1 (t,x) \Vert_{2n} $.

Next, we write $I_2(t')-I_2(t)$ as the sum of
$J_1=\int_t^{t'} p_{t'-s}*dM_s$
and $J_2=\int_0^t (p_{t'-s}-p_{t-s})*dM_s$.
Applying the argument for bounding $ I_1 $ to bound the term $ J_1 $,
we obtain
\begin{align*}
	\Vert (J_1)^2 \Vert_{n}
	\leq
	C ( \eps^\alpha |t'-t|^\frac{\alpha}{2} e^{a\eps|x|}  )^2 \;.
\end{align*}
As for $ J_2 $, applying  Lemma \ref{lem:BDG} using
$
	(p_{t'-s}-p_{t-s})^2 \le |p_{t'-s}-p_{t-s}|\, (p_{t'-s}+p_{t-s})
$
followed by the estimate (see for instance \cite[(A.7)]{DemboTsai})
\begin{align*}
	|p_{t'}(x) - p_t(x)|
	\leq
	C (1\wedge t^{-\frac12-\alpha}) \,(t'-t)^\alpha
\end{align*}
one obtains the desired bound
$ \Vert J_2\Vert_{2n}^2 \leq C \eps^{2\alpha} |t'-t|^\alpha e^{2a\eps|x|} $.
Combining all these bounds completes the proof of the proposition.
%
%
\end{proof}

\begin{proof}[of Proposition~\ref{prop:stepEst}]
With $ Z^*_t(x) = \frac{1}{2 \sqrt \eps} Z_t(x) $,
similar to \eqref{eq:iter} we have
\begin{align}\label{eq:I1*}
	\Vert (Z^*_t(x))^2 \Vert_{2n}
	&\leq
	C(I_1^*(t,x))^2 + C \e \int_0^t (t-s)^{-1/2} \big( p_{t-s} * \Vert (Z^*_t)^2 \Vert_{n} \big)(x) ds
\end{align}
where $ I_1^*(t,x) = \frac{1}{\sqrt \e} (p_t * e^{-\sqrt{\e}|\Cdot|})(x) $.
With $ p_t \leq C/\sqrt{t} $, we have
$ |I_1^*(t,x)| \leq \frac{C}{\sqrt{\e t}} $.
Using this in \eqref{eq:I1*} yields
\begin{align}\label{eq:I1*:}
	\Vert (Z^*_t(x))^2 \Vert_{2n}
	&\leq
	\frac{C}{\sqrt{\e^2 t}} (p_t * e^{-\sqrt{\e}|\Cdot|})(x)
	+ C \e \int_0^t (t-s)^{-1/2} \big( p_{t-s} * \Vert (Z^*_t)^2 \Vert_{n} \big)(x) ds.
\end{align}
Now, iterate this equation using the semi-group property $ p_{s} * p_{s'} = p_{s+s'} $ to obtain
\begin{align*}
	\Vert (Z^*_t(x))^2 \Vert_{2n}
	&\leq
	\frac{C}{\sqrt{\e^2 t}} (p_t * e^{-\sqrt{\e}|\Cdot|})(x)
	+
	\sum_{j=1}^\infty C^j I_j(\e^{2}t) (p_t * e^{-\sqrt{\e}|\Cdot|})(x)
\end{align*}
where $ I_j(T) = \int_{\Delta_j(T)} (S_1\cdots S_{j+1})^{-1/2} dS_1\cdots dS_j $
and $ \Delta_j(T) := \{ (S_1,\ldots,S_{j+1}) \in (0,\infty)^{j+1} : S_1+\ldots+S_{j+1} = T \} $.
With $ I_{(j)}(T) = T^{(j-1)/2} \Gamma(1/2)^{j+1}/\Gamma((j+1)/2) $,
we have $ \sum_{j=1}^\infty C^j I_j(\e^{2}t) \leq C $, and consequently
\begin{align}\label{eq:Z*bd}
	\Vert (Z^*_t(x))^2 \Vert_{2n}
	&\leq
	\frac{C}{\sqrt{\e^2 t}} (p_t * e^{-\sqrt{\e}|\Cdot|})(x) \;.
\end{align}
Further using $ (p_t * e^{-\sqrt{\e}|\Cdot|})(x) \leq \frac{C}{\sqrt{\e^2 t}} $,
we conclude the desired bound \eqref{eq:momd}.

Turning to proving \eqref{eq:momdx},
similar to \eqref{eq:I1*:} we have
\begin{align}
	\notag
	\Vert (Z^*_t(x))^2 \Vert_{2n}
	&\leq
	\frac{C}{\sqrt{\e^2 t}}
	\Big| (p_t * e^{-\sqrt{\e}|\Cdot|})(x)  -(p_t * e^{-\sqrt{\e}|\Cdot|})(x') \Big|
\\
	\label{eq:I1*::}
	&+
	C |\e(x-x')|^{2\alpha} \e^{1-2\alpha} \int_0^t (t-s)^{-1/2-\alpha} \big( p_{t-s} * \Vert (Z^*_t)^2 \Vert_{n} \big)(x) ds.
\end{align}
Using $ |p_t(x+y)-p_t(y)| \leq C t^{-1/2-\alpha} |y|^{2\alpha} $
and \eqref{eq:Z*bd} to bound the respective terms on the RHS,
we conclude the desired bound \eqref{eq:momdx}.
\end{proof}

%

\section{Identifying the limit, proof of Proposition~\ref{prop:unique}}
\label{sect:unique}

In order to identify the limit of $\mathcal Z^\eps$, we recall
(for instance \cite[Proposition~4.11]{BG})
 that the mild solution $\mathscr Z$ to \eqref{e:SHE}
 with initial condition $\mathscr Z^\ic$ is equivalent to
 the {\it unique} solution of the martingale  problem
 with initial condition $\mathscr Z^\ic$, provided that
$ \Vert \mathscr Z^\ic(X) \Vert_2
	\le C e^{a |X|} $ for some $C,a>0$.
Also recall that a $C(\R_+,C(\R))$ valued process
$Z$  
 is said to {\it solve the martingale problem}
with initial condition $\mathscr Z^\ic$ if
$Z_0 =\mathscr Z^\ic$ 
in distribution, and
for all $\bar T>0$, there exists $a\ge 0$ such that
\begin{equ} \label{e:mart-uni}
\sup_{T\in[0,\bar T]} \sup_{X\in\R} e^{-a|X|}
	\E \big( Z_T(X)^2\big) <\infty
\end{equ}
and  for all $\varphi\in\mathcal C_c^\infty(\R)$,
\begin{equs}
	\label{eq:Nt}
	N_T(\varphi )
	& \eqdef
	(Z_T,\varphi )- (Z_0,\varphi)
	- \frac12 \int_0^T (Z_S,\varphi'')\,dS
\\
	\label{eq:LambdaT}
	\Lambda_T(\varphi)
	& \eqdef
	N_T(\varphi)^2 -\int_0^T (Z_S^2,\varphi^2 )\,dS
\end{equs}
are local martingales.
Here, $(\varphi,\psi)\eqdef \int_\R \varphi(X)\psi(X)\,dX$.

\begin{proof} [of Proposition~\ref{prop:unique}]
By \eqref{e:uniform}, 
any limit point of the family $\mathcal Z^\eps$ satisfies \eqref{e:mart-uni}.
Since $ Z^\e_0 \Rightarrow \mathscr Z^\ic $,
the initial condition of the martingale problem is also satisfied
for any limit point.

Define for all $t\in[0,\eps^{-2}\bar T]$,
$\varphi\in\mathcal C_c^\infty(\R)$
\begin{equ}
	(Z_{t},\varphi)_\eps
	\eqdef \eps_j\sum_{x\in\Z} \varphi(\eps_j x) Z_{t}(x) \;.
\end{equ}
Recall that $\eps_j$ was introduced in \eqref{e:calZ} as $\eps_j =2j\eps$.

Consider the microscopic analogs of \eqref{eq:Nt}--\eqref{eq:LambdaT} as
\begin{equs}
	\label{eq:Nte}
	N^\eps_T(\varphi) & \eqdef
	(Z_{\eps_j^{-2} T},\varphi)_\eps - (Z_0,\varphi)_\eps
	-\frac12 \int_0^{\eps_j^{-2}T}
	(\Delta Z_{s},\varphi)_\eps \,ds
\\
	\Lambda_T^\e(\varphi) &\eqdef
	N^\e_T(\varphi)^2
	- \langle N^\eps_T(\varphi)\rangle \;.
\end{equs}
Indeed, by Proposition~\ref{prop:HC},
$ N^\eps_T(\varphi) $ and hence $ \Lambda^\eps_T(\varphi) $ are martingales.
Further applying \eqref{e:ASEPbrac} to calculate $ \langle N^\eps_T(\varphi)\rangle $
and using the factor $ \mathbf{1}_{\{x=y\}}$
to re-write a double sum as a single sum over lattice sites,
we obtain the following expression for $ \Lambda_T^\e(\varphi) $:
\begin{equs}\label{eq:LambdaTe}
	\Lambda^\eps_T(\varphi) &\eqdef
	N_T^\eps(\varphi)^2
		- \eps_j^2 \int_0^{\eps_j^{-2}T} (Z_{s}^2,\varphi^2)_\eps \,ds
		+R_1^\eps (\varphi) +R_2^\eps (\varphi) +R_3^\eps (\varphi)
\end{equs}
where 
\begin{equs}
R_1^\eps(\varphi) &\eqdef
\eps_j^2 \,(\tfrac{2j}{[2j]_q}-1)
	\int_0^{\eps_j^{-2}T} (Z_{s}^2,\varphi^2)_\eps \,ds \;, \\
R_2^\eps(\varphi) &\eqdef  \frac{\eps_j}{[2j]_q}
\int_0^{\eps_j^{-2} T} ( \nabla^- Z_{s}\nabla^+ Z_{s},\varphi^2)_\eps \,ds \;, \\
R_3^\eps(\varphi) &\eqdef o(\eps^2)
\int_0^{\eps_j^{-2} T} ( Z_s^2,\varphi^2)_\eps \,ds \;.
\end{equs}

In \eqref{eq:Nte},
applying summation by part yields
$ (\Delta Z_{s},\varphi)_\eps = (Z_{s}, \Delta \varphi)_\eps $.
Further, as $ \varphi \in C^\infty_c(\R) $,
we have that $ \e_j^{-2} \Delta \varphi $ converges uniformly to $ \varphi'' $.
%
%
by comparing the expressions as in \eqref{eq:Nt}--\eqref{eq:LambdaT}
and \eqref{eq:Nte}--\eqref{eq:LambdaTe},
it clearly suffices to prove that $ \E(R^\e_i(\varphi))^2 \to 0$, for $i=1,2,3$.
By the uniform bound \eqref{e:uniform} on $Z$,
with $|\frac{2j}{[2j]_q}-1|\le C\eps$,
we clearly have $\E (R_{i}^\eps(\varphi)^2 ) \to 0$, for $ i=1,3 $.
To control $ R_2^\eps(\varphi)^2 $,
we follow \cite{BG} by using the ``key estimate'' as in Lemma~\ref{lem:key-est} in the following.
Indeed, letting $ \mathcal F_{t} \eqdef \sigma( Z_s(x) : x\in\Z, s\leq t ) $ 
denote the canonical filtration and let
\begin{align}\label{eq:U}
	U^\e(y,s,s') \eqdef
	\E ( \nabla^- Z_{s}(y)\nabla^+ Z_{s}(y) \,|\,  \mathcal F_{s'} ),
\end{align}
with $ R^\e_2(\varphi) $ defined as in the preceding, we have
\begin{equs}
\E (R_2^\eps(\varphi)^2 )& =
\frac{2\eps_j^4}{[2j]_q^2}
 \int_0^{\eps_j^{-2}T} \!\!\!\! ds \int_0^s ds'
	 \sum_{x,y\in\Z} \varphi(\eps x)^2 \varphi(\eps y)^2
	\E \Big(
	\nabla^- Z_{s'}(x)\nabla^+ Z_{s'}(x) U^\e(y,s,s') \Big) \;.
\end{equs}
With $ |\nabla^\pm Z_t(x)| \le C\eps^{\frac12} Z_t(x) $,
we further obtain
\begin{equs}
	\E (R_2^\eps(\varphi)^2 )
 & \le C \eps^5
 \int_0^{\eps_j^{-2}T} \!\!\!\!  ds \int_0^s ds'
 \sum_{x,y\in\Z} \varphi(\eps x)^2 \varphi(\eps y)^2
	\E \Big(
	Z_{s'}(x)^2 U^\e(y,s,s') \Big) \;. 		
	\label{e:double-E}
\end{equs}

Note if we simply use $ |\nabla^\pm Z_t(x)| \le C\eps^{\frac12} Z_t(x) $
to bound $ U^\e(y,s,s') $ as
	$ |U^\e(y,s,s')| \leq \e C Z^2_s(y) $,
and insert this bound into \eqref{e:double-E},
the resulting bound on $ \E(R^\e_2(\varphi)^2) $ is of order $ O(1) $,
(since the change of time and space variables to macroscopic variables gives $\eps^{-6}$),
which is insufficient for our purpose.
To obtain the desired bound $ \E(R^\e_2(\varphi)^2) \to 0 $,
we utilize the smoothing effect of the conditional expectation $ \E(\Cdot|\mathcal F_{s'}) $ in \eqref{eq:U} to show the following
\begin{lemma} \label{lem:key-est}
For all $\bar T>0$, $\delta>0$, there are constants $ a,C>0 $ such that
\begin{equ} \label{e:key-est}
	\sup_{x\in\Z} e^{-a\eps|x|}
	\E | U^\e(x,t,s) |
	\le C
	 \eps^{\frac32-\delta}(\e^{2}(t-s))^{-\frac12}
\end{equ}
for all $\sqrt\eps \le \e^{2} s< \e^{2}t \le \bar T$ and all $\eps>0$.
\end{lemma}
With this, $ \E(R^\e_2(\varphi)^2) \to 0 $ follows by standard argument
as in \cite[Proof of Proposition~4.11]{BG}.
We omit the details here and prove only Lemma~\ref{lem:key-est}.
%
%
\end{proof}

Proving Lemma~\ref{lem:key-est}
requires a certain integral identity on the heat kernel $ p_t(x) $ as in \cite[Lemma~A.1]{BG}.
Here, to shed light on the underlying structure of this identity,
we state and prove the following more general identity.

\begin{lemma} \label{lem:identity}
Let $p_t(x)$ be the transition probability of the continuous time symmetric simple random walk on $\Z^d$,
with the convention $p_t(x)=0$ for $t<0$. Then one has
\begin{equ} \label{e:d-dim-id}
\frac{1}{d} \sum_{x\in \Z^d} \sum_{n=1}^d \int_{-\infty}^\infty
	 \nabla_n p_{t+s} (x+y) \nabla_n p_{t+s'} (x+y') \,dt
= p_{|s-s'|}(y-y') \;, 
\end{equ}
for all $s,s'\in\R$ and $y,y'\in \R^d$,
where
\[
\nabla_n f(x_1,\dots,x_d)\eqdef f(x_1,\dots,x_n+1,\dots,x_d)
	-f(x_1,\dots,x_d) \;.
\]
\end{lemma}

\begin{proof}
Let $\mathscr F_x,\mathscr F_t  $ denote the Fourier transform operators
in the spatial variable and time variable respectively,
and let $\mathscr F$ denote the Fourier transform operator in both variables.
Since $p$ solves $\partial_t p = \frac{1}{2d} \Delta p$ with initial condition $\mathbf{1}_{x=0}$, and $e^{ik\cdot x}$
is the eigenfunction of $ \frac{1}{2d} \Delta$ with eigenvalue
$\lambda_k \eqdef \frac{1}{d} \sum_{n=1}^d (\cos k_n-1)$,
we have
\begin{equ}
(\mathscr F p)\, (\omega,k)    =
	\frac{1}{-\lambda_k+ i\omega}  \;.
\end{equ}
The LHS of \eqref{e:d-dim-id} can be written as
$\frac{1}{d} \sum_{n=1}^d (\nabla_n p \,\hat{*} \widetilde{ \nabla_n p})_{s-s'}(y-y') $
where 
\[
\widetilde{ \nabla_n p}  (t,x)=\nabla_n p(-t,-x)
\]
 denotes reflected function,
and $ \hat{*} $ denotes the space-time convolution, as
\[
 (f \hat{*} g)_s(y) \eqdef 
\int_{-\infty}^{\infty} \sum_{x\in\Z^d} f_{t+s}(y+x) g_{-t}(-x) dt \;.
\]
Therefore the Fourier transform of the LHS of \eqref{e:d-dim-id}
is equal to
\begin{equ}
\frac{1}{d} \sum_{n=1}^d
\Big| \frac{e^{ik_n}-1}{-\lambda_k+ i\omega}   \Big|^2
= \frac{1}{d} \sum_{n=1}^d
	\frac{2-2\cos k_n}{\lambda_k^2+ \omega^2}  \;. \label{e:idenLHS}
\end{equ}
On the other hand, for the RHS of \eqref{e:d-dim-id}, one has
$
(\mathscr F_x p_{|t|})\, (k)  = e^{\lambda_k |t|}
$.
Further take Fourier transform in $t$, one has
\footnote{Here we use the fact that for any $a>0$,
$\mathscr F_t e^{-a|t|} =\frac{2a}{a^2+\omega^2}$.}
\[
\mathscr F_t (\mathscr F_x p_{|\Cdot|}(\Cdot))\, (\omega,k)  = \frac{-2\lambda_k}{\lambda^2+\omega^2} \;,
\]
which is equal to \eqref{e:idenLHS}.
\end{proof}

\begin{remark}
A continuous version of \eqref{e:d-dim-id} for $d=1$
is stated in the recent paper \cite[Proof of Lemma~6.11]{KPZJeremy} (up to a factor $2$
on the LHS because the heat operator is defined as $\partial_t-\Delta$ therein),
and is used to show that the logarithmically divergent renormalization constants
add up to a finite constant $c$ and if the KPZ equation is only spatially regularized then $c=0$.
\end{remark}

Now, setting $d=1$, $s=s'=0$ and $y,y'\in\{0,-1\}$, one recovers \cite[Lemma~A.1]{BG}
\begin{equ}
	\sum_x \int_0^\infty \nabla^+ p_t(x) \nabla^- p_t(x)  \,dt =0 \;,
	\label{e:identity0}
\end{equ}
and, by using also the Cauchy-Schwartz inequality, we also obtain \cite[Lemma~A.2]{BG}
\begin{equ}
\sum_x \int_0^\infty \big|\nabla^+ p_t(x) \nabla^- p_t(x) \big| \,dt
< \!\!\! \prod_{\sigma\in\{+,-\}}
\Big( \sum_x  \int_0^\infty (\nabla^\sigma p_t(x))^2 \,dt  \Big)^\frac12
=1\cdot  1=1 \;.
\label{e:identity1}
\end{equ}



\begin{proof}[of Lemma~\ref{lem:key-est}]
The proof follows the same argument as in \cite[Lemma~4.8]{BG}.

Let $ N_s^t(x) \eqdef \int_0^s p_{t-\tau} * dM_\tau $
so that $Z_t(x)=I_t(x)+N_t^t(x)$ where $I_t =p_t*Z_0$.
For $s\le r\le t$, one has
\begin{equs} \label{e:dNdN}
\E \Big(
		\nabla^- N_{r}^t(x)\nabla^+ N^t_{r}(x)
	\,\big|\, \mathcal F_{s}\Big)
&=\nabla^- N_{s}^t(x)\nabla^+ N^t_{s}(x) \\
 & \qquad + \E\Big(  \int_s^r K_{t-\tau} * d\langle M(\cdot),M(\cdot)\rangle_\tau (x)
\,\Big|\, \mathcal F_s\Big)
\end{equs}
where
\begin{equ} \label{e:def-K}
	K_t(x) \eqdef \nabla^+ p_t(x) \nabla^- p_t(x) \;.
\end{equ}
With $ U^\e(y,t,s) $ defined as in \eqref{eq:U} and with $\E(  N_r^t(x) | \mathcal F_s )=N_s^t(x)$,
one has by \eqref{e:dNdN}
\begin{equs}
	U^\e(y,t,s)
	&= \nabla^- I_t(x) \nabla^+ I_t(x)
	+ \nabla^- I_t(x) \nabla^+ N_s^t(x)
	+ \nabla^- N_s^t(x) \nabla^+ I_t(x)
\\
	&+
	\nabla^- N_{s}^t(x)\nabla^+ N^t_{s}(x)
	+ \E\Big(  \int_s^t K_{t-\tau} * d\langle M(\Cdot) \rangle_\tau (x)
	\,\Big|\, \mathcal F_s\Big) \;.   \label{e:five-terms}
\end{equs}

We bound the $L^1$-norms (i.e. $\E|\Cdot|$) of the terms on the RHS.
For the first four terms,
by the Cauchy-Schwartz inequality one needs only to show
\begin{align}\label{eq:IN:bd}
	\E(\nabla^\pm I_t(x))^2 \;, \; \E(\nabla^\pm N^t_s(x))^2 \leq C \eps^{\frac12}(t-s)^{-\frac12}e^{2a\eps |x|}.
\end{align}
To bound $ \nabla^\pm I $, we use \eqref{e:init-uniform} to obtain
\begin{equs}
	\E \Big( \big(\nabla^\pm I_t(x) \big)^2 \Big)
	&=
	\sum_{y,y'} \nabla^\pm p_t(y)\nabla^\pm p_t(y')
	\,\E \big( Z_0(x-y) Z_0(x-y')\big)
\\
	&\le
	C e^{2a\eps |x|} \Big( \sum_y  \nabla^\pm p_t(y)\,e^{a\eps |y|}  \Big)^2 \;.
\end{equs}
Using \cite[(A.26)]{DemboTsai} with $ v=1 $,
we bound the RHS by $ C e^{2a\eps |x|} t^{-1} $.
Further expressing $ t^{-1} $ as $ t^{-1/2} t^{-1/2} $,
and applying $ t^{-\frac12} < (t-s)^{-\frac12} $
and $ t^{-1/2} \leq \e^{3/4} $ (since we assume $ \e^{2} t \geq \e^{1/2} $),
we obtain desired bound on $ \E(\nabla^\pm I)^2 $ as in \eqref{eq:IN:bd}.
Turning to bounding $ \E(\nabla^\pm N)^2 $, one has
\begin{align*}
	\E\Big( \big(\nabla^\pm N_s^t(x) \big)^2\Big)
	 &=
	\E \int_0^s \sum_y \big(\nabla^\pm p_{t-\tau} \big)^2
	* d\langle M \rangle_\tau \;
\\
	&\leq	
	C
	\int_0^s
	\Big( \sup_y |\nabla^+ p_{t-\tau}(y)| \Big)
	\Big( |\nabla^- p_{t-\tau}|*
	\E|\tfrac{d~}{d\tau}\langle M \rangle_\tau| \Big)(x) d\tau \;.
\end{align*}
By \eqref{e:another-brac} and the uniform bound \eqref{e:uniform},
one has $ \E|\frac{d}{d\tau}\langle M(y) \rangle_\tau| \le C\eps e^{a\eps |y|}$;
we then apply the estimates \cite[(A.26), (A.28)]{DemboTsai} with $ v=1 $
to obtain
\begin{align*}
	\E \Big( \big(\nabla^\pm N_s^t(x) \big)^2\Big)
	\leq
	C \e e^{2a|x|} \int_0^s (t-\tau)^{-3/2} ds.
\end{align*}
Upon integrating over $ \tau $,
we obtain the desired bound on $ \E(\nabla^\pm N)^2 $ as in \eqref{eq:IN:bd}.

To bound the last term on the RHS of \eqref{e:five-terms},
we use the explicit expression of the predictable quadratic variation \eqref{e:ASEPbrac}
to re-write the last term on the RHS of \eqref{e:five-terms} as
$ I_1 + I_2 +  I_3 $ where
\begin{equs}
	I_1 (s,t,x) &\eqdef \frac{4\eps j^2}{[2j]_q}\sum_y \int_s^t K_{t-\tau}(x-y)\,
		\E (Z_\tau(y)^2 \,|\,\mathcal F_s) \,d\tau \;,
\\
	I_2 (s,t,x) &\eqdef -\frac{1}{[2j]_q} \sum_y \int_s^t K_{t-\tau}(x-y)\,
		\E (\nabla^-Z_\tau(y)\nabla^+Z_\tau(y) \,|\,\mathcal F_s) \,d\tau  \;,
\\
	I_3 (s,t,x) &=  o(\eps) \sum_y \int_s^t K_{t-\tau}(x-y)\,
		\E (Z_\tau(y)^2 \,|\,\mathcal F_s) \,d\tau  \;.
\end{equs}
Indeed, $  0\leq I_3 \leq I_1 $ for all $ \e $ small enough,
so we drop $ I_3 $ in the following.

To bound $I_1$ we apply the  identity 
\eqref{e:identity0} to obtain
\begin{equs} \ 
	I_1 (s,t,x)
	&=\frac{4\eps j^2}{[2j]_q}  \sum_y \int_s^t K_{t-\tau}(x-y)\,
	\E (Z_\tau(y)^2-Z_t(x)^2 \,|\,\mathcal F_s) \,d\tau
\\
	&\quad
	+\frac{4\eps j^2}{[2j]_q}  \E (Z_t(x)^2 \,|\,\mathcal F_s)
	\sum_y \int_{t-s}^\infty K_\tau(x-y) \,d\tau \;.
\end{equs}
Hence $ |I_1(s,t,x)| \leq C (I_{11}(s,t,x)+I_{12}(s,t,x)) $, where
\begin{align}
	\label{eq:I11}
	I_{11}(s,t,x)
	&\eqdef
	\e \sum_y \int_s^t K_{t-\tau}(x-y)\,
		\E |Z_\tau(y)^2-Z_t(x)^2 | \,d\tau
\\
	\label{eq:I12}
	I_{12}(s,t,x)
	&\eqdef
	\e \E (Z_t(x)^2 )
	\int_{t-s}^\infty \sum_y K_\tau(x-y) \,d\tau \;.
\end{align}
With $ K $ defined as in the preceding,
applying \cite[(A.26), (A.28)]{DemboTsai} with $ v=1 $,
we obtain $ \sum_{y} |K_\tau(x-y)| \leq C(1\wedge \tau^{-3/2}) $.
Using this and the uniform bound \eqref{e:uniform} in \eqref{eq:I12},
we obtain the desired bound on $ I_{12} $ as
\begin{equ} \label{e:I1-2nd}
	\E |I_{12}(s,t,x)|
	\leq
		C \e e^{a\eps |x|} \int_{t-s}^\infty \tau^{-\frac32} \,d\tau
	\leq
		C \e e^{a\eps |x|} (t-s)^{-\frac12} \;.
\end{equ}
Next, the idea of controlling $ I_{11} $
is to use the fact that $ K_{t-\tau}(x-y) $
concentrates on values of $ (\tau,y) $ which are close to $ (t,x) $,
and that, thanks to the H\"older estimates \eqref{e:Holderx}--\eqref{e:Holdert},
$ |Z_\tau(y)^2 -Z_t(x)^2| $ is small when $ (\tau,y) \approx (t,x) $.
More precisely,
with
\begin{align*}
	|Z_\tau(y)^2-Z_t(x)^2|
	\leq
	(Z_\tau(y)+Z_t(x)) \big( |Z_\tau(y)-Z_t(y)| + |Z_t(y)-Z_t(x)| \big)
\end{align*}
we use the Cauchy--Schwarz inequality and the H\"older estimates
\eqref{e:Holderx}--\eqref{e:Holdert} for $ \alpha = \frac12 -\delta $
to obtain
\begin{align*}
	\E |Z_\tau(y)^2-Z_t(x)^2|
	\le
	C \e^{\frac12 -\delta} e^{a\eps (|x|+|y|)} \Big( |y-x|^{\frac12-\delta}+
	(|t-\tau|\vee 1)^{\frac14-\delta/2} \Big) \;.
\end{align*}
Inserting this into \eqref{eq:I11},
after the change of variables $ t-\tau\mapsto \tau $ and $ x-y\mapsto y$,
we arrive at
\begin{equs}
	\E |I_{11}(s,t,x)|
	\leq
	C &\eps^{\frac32-\delta} e^{a\eps |x|}
	\int_0^{\eps^{-2} \bar T}
	\Big( \sup_y |\nabla^+ p_{\tau} (y)| \Big)
\\
	&
	\quad\quad
	\times
	\Big(
		\sum_y  | \nabla^- p_{\tau} (y) |
		\,e^{a\eps |y|}
		( |y|^{\frac12}+(|\tau|\vee 1)^{\frac14} )
	\Big)
	\,d\tau \;.
\end{equs}
Further using \cite[(A.26), (A.28)]{DemboTsai} with $ v=1 $,
to bound the terms within the integral,
we obtain
\begin{equs}
	\E |I_{11}(s,t,x)|
	\leq
	C \eps^{\frac32-\delta} e^{a\eps |x|}
	\int_0^{\eps^{-2} \bar T}
	(1\wedge \tau^{-1}) \tau^{-1/4} d\tau
	\leq
	C \eps^{\frac32-\delta} e^{a\eps |x|}  \;.
\end{equs}
With $ (t-s)^{-1/2} \leq t^{-1/2} \leq {\bar T}^{-1/2} \e^{-1} $,
the desired bound
$
	\E |I_{11}(s,t,x)|
	\leq
	C \eps^{\frac12-\delta} e^{a\eps |x|} (t-s)^{-1/2}
$
follows.

So far we have obtained the desired bounds on all the terms on the RHS of \eqref{e:five-terms}
except for the term $ I_2 $ from the last term in \eqref{e:five-terms};
but $I_2$ contains the same conditional expectation
on the LHS of \eqref{e:five-terms}.
Define $A_t$ to be the LHS of \eqref{e:key-est}.
Collecting the bounds for  the terms in \eqref{e:five-terms},
then multiplying both sides by $e^{-3a\eps|x|}$ and taking supremum, one has
\begin{equ} \label{e:iterate-EE}
A_t \le C\eps^{\frac12-\delta}(t-s)^{-\frac12}
+
\sum_y \int_s^t  |K_{t-\tau} (y)| e^{a\eps|y| } A_\tau \,d\tau \;,
\end{equ}
where a change of variable $x-y\mapsto y$ is preformed.
The desired estimate \eqref{e:key-est} now
follows by iterating \eqref{e:iterate-EE} as in \cite[Lemma~4.8]{BG}.
%
\end{proof}

\section{Remarks on ASIP($q,k$)}

The asymmetric inclusion process with parameters $q,k$ (ASIP($q,k$) for short) is introduced in \cite{CGRS2015}, which also enjoy a self-duality property similar to that of ASEP($q,j$). 
In this section we apply our methods
in Section~\ref{sect:SHE} to derive a microscopic Cole-Hopf transformation of ASIP($q,k$), 
and discuss the possibility of showing convergence to the KPZ equation.
Following \cite{CGRS2015}, we consider the process on 
the finite lattice $ \Lambda_L =\{1,\dots,L\} $.

\begin{definition} (ASIP($q,k$) on $ \Lambda_L $.)
Let $q\in(0,1)$ and  $k\in\R_+$ be a  positive real number.
Denote by  $\widetilde\eta(x) \in \N$
the occupation variable, i.e. the number of particles at site $ x\in \Lambda_L $.
Note that $\widetilde\eta(x)$ can be any non-negative integer.
The ASIP($q,k$) is a continuous-time Markov process
on the state space
$\N^{\mathbf L}$
defined by: at any given time $ t\in[0,\infty) $,
a particle jumps from site $ x $ to site $ x+1 $
at rate
%
\begin{align*}
	\widetilde{c}_q^+(\widetilde{\eta},x)
	&= \frac{1}{2[2k]_q} q^{\widetilde{\eta}(x)-\widetilde{\eta}(x+1)+(2k-1)}
	[\widetilde{\eta}(x)]_q
	 [2k+\widetilde{\eta}(x+1)]_q
\end{align*}
and from site $ x+1 $ to site $ x $ at rate
\begin{align*}
	\widetilde{c}_q^-(\widetilde{\eta},x)
	&= \frac{1}{2[2k]_q} q^{\widetilde{\eta}(x)-\widetilde{\eta}(x+1)-(2k-1)}
	[2k + \widetilde{\eta}(x)]_q [\widetilde{\eta}(x+1)]_q
\end{align*}
independently of each other.
\end{definition}

As in Definition~\ref{def:ASEPqj}
we define the {\it centered} occupation variable
$\eta(x) \eqdef \widetilde \eta(x)-k \in \N -k$
and the corresponding jumping rate
\begin{equ}\label{e:ratesASIP}
c_q^\pm (\eta,x)
	=  \frac{1}{2[2k]_q} q^{\eta(x)-\eta(x+1) \pm (2k-1)}
	\,[\eta(x) \mp k]_q \, [\eta(x+1) \pm k]_q    \;.
\end{equ}
Define $\eta^{x,y}$ in the preceding.
With these notations, the ASIP$(q,k)$ has the generator
\begin{equ}\label{genZASIP}
	({\cal L}f)(\eta) =\sum_{x\in \Lambda_L} ({\cal L}_{x,x+1}f)(\eta)
\end{equ}
where
$
({\cal L}_{x,x+1}f)(\eta)
 =  c^+_q (\eta, x) \, (f(\eta^{x,x+1}) - f(\eta) )
 + c^-_q (\eta, x) \, (f(\eta^{x+1,x}) - f(\eta) ).
$
\begin{remark}\label{rmk:ASEPtoASIP}
By comparing \eqref{e:rates} and \eqref{e:ratesASIP},
we find that the \emph{generator} of ASEP($ q,k $) is
converted to that of ASIP($ q,j $) by letting $ j\mapsto -k $.
\end{remark}

The article \cite{CGRS2015} raised up the following question.
\begin{question}
Can the ASIP$(q,k)$ be constructed on the entire $\Z$?
\end{question}

Define the processes $h$ and $Z$ in the same way as in \eqref{e:def-h}
and \eqref{e:Z-defn}, with respect to the ASIP($q,k$) occupation configuration
$\eta$. Set
\begin{align}\label{nu-ASIP}
	\nu \eqdef \big( \tfrac{[4k]_q}{2[2k]_q} -1 \big) / \ln q\;.
\end{align}
Parallel with Proposition~\ref{prop:HC},
we have

\begin{proposition}   
	We have that
	\begin{equ}\label{dualityASIP}
		d Z_t(x) = \tfrac12 \Delta Z_t(x)\,dt + dM_t(x)
	\end{equ}
	where $ M_\Cdot(x) $, $ x\in\mathbf L $, are martingales.
\end{proposition}

\begin{proof}
Proceeding as \eqref{e:Omega} and \eqref{e:M}
we have that
\begin{equ}
	dZ_t(x) =  \Omega Z_t(x)  \,dt
			+ dM_t(x)  \;,
		\qquad
	\Omega= \sum_{\sigma =\pm}(q^{2\sigma}-1)
		c_q^\sigma (\eta,x)
			+\nu \ln q
\end{equ}
where $c_q^\pm$ is now defined as \eqref{e:ratesASIP},
and the martingales $\{M_t(x)\}_{x\in \mathbf L}$  are defined as
\eqref{e:M} with the ASIP($q,k$) rates \eqref{e:ratesASIP}.
To compute $ \Omega Z_t(x) $, by Remark~\ref{rmk:ASEPtoASIP},
we simply perform the substitution $ j \mapsto -k $
in the in the proof of Proposition~\ref{prop:HC}\ref{enu:HCdr},
whereby obtaining 
$ [2k]_q  (\Omega -\nu \ln q )=[2k]_q \Delta Z(x) -[4k]_q+2[2k]_q $.
With this and \eqref{nu-ASIP}, the statement \eqref{dualityASIP} follows.
\end{proof}

We turn to consider the bracket process of  ASIP($q,k$).
As in the proof of Proposition~\ref{prop:HC}\ref{enu:HCMG},
we compute
\begin{equs}
\frac{d}{dt}&\langle M(x),M(y) \rangle _t
	=\frac{\mathbf{1}_{\{x=y\}}}{2[2k]_q} Z_t(x)^2   \\
& \times \Big( q \, (q^{2\eta(x)}-q^{2k} )\,(q^{2k}-q^{-2\eta(x+1)} )
	 -q^{-1}\,(q^{2\eta(x)}-q^{-2k} )\,(q^{-2k} + q^{-2\eta(x+1)} )
	 	\Big)
\end{equs}
where $q=e^{-\sqrt{\eps}}$.
But the occupation variable $\eta(x)$ is unbounded ASIP($q,k$),
so the argument of Taylor expansion in $\sqrt\eps$ in
the proof of Proposition~\ref{prop:HC}\ref{enu:HCMG}
is not useful.

\begin{question}
Does ASIP$(q,k)$ converge to the KPZ equation under the same scaling as studied in ASEP($ q,j $)?
\end{question}

%
%

\bibliographystyle{./Martin}
\bibliography{./refs}

\end{document}